\documentclass[letterpaper,10pt]{amsart}

\usepackage[all]{xy}                      %
  
\CompileMatrices                            

\UseTips                                    

\input xypic
\usepackage[bookmarks=true]{hyperref}       

\usepackage{amssymb,latexsym,amsmath,amscd}
\usepackage{xspace}
\usepackage{color}
\usepackage{graphicx}
\usepackage{dsfont}

\reversemarginpar

\theoremstyle{plain}
\newtheorem{theorem}{Theorem}[section]
\newtheorem*{theorem*}{Theorem}
\newtheorem{proposition}[theorem]{Proposition}
\newtheorem{corollary}[theorem]{Corollary}
\newtheorem{lemma}[theorem]{Lemma}

\theoremstyle{definition}

\newtheorem{remark}[theorem]{Remark}
\newtheorem{example}[theorem]{Example}

\newcommand{\enm}[1]{\ensuremath{#1}}          %

\newcommand{\cal}[1]{\mathcal{#1}}

\newcommand{\ZZ}{\enm{\mathbb{Z}}}

\newcommand{\PP}{\enm{\mathbb{P}}}

\newcommand{\Ff}{\enm{\cal{F}}}

\newcommand{\Ii}{\enm{\cal{I}}}

\newcommand{\Oo}{\enm{\cal{O}}}

\newcommand{\Ss}{\enm{\cal{S}}}

\renewcommand{\phi}{\varphi}
\renewcommand{\theta}{\vartheta}
\renewcommand{\epsilon}{\varepsilon}

\renewcommand{\to}[1][]{\xrightarrow{\ #1\ }}

\newcommand{\old}[1]{}

\begin{document}

\title[Dependent subsets]
{Dependent subsets of embedded projective varieties}
\author{Edoardo Ballico}
\address{Dept. of Mathematics\\
 University of Trento\\
38123 Povo (TN), Italy}
\email{ballico@science.unitn.it}
\thanks{The author was partially supported by MIUR and GNSAGA of INdAM (Italy).}

\begin{abstract}
Let $X\subset \mathbb {P}^r$ be an integral and non-degenerate variety. Set $n:= \dim (X)$. Let $\rho (X)''$ be the maximal integer
such that every zero-dimensional scheme $Z\subset X$ smoothable in $X$ is linearly independent. We prove that $X$ is linearly normal if $\rho (X)''\ge
\lceil (r+2)/2\rceil$ and that $\rho (X)'' < 2\lceil (r+1)/(n+1)\rceil$, unless either $n=r$ or $X$ is a rational normal curve.
\end{abstract}

\maketitle

\section{Introduction}

Let $X\subset \PP^r$ be an integral and non-degenerate variety defined over an algebraically closed field with characteristic
zero. Set
$n:=
\dim X$. We recall that a zero-dimensional scheme $Z\subset X$ is said to be \emph{smoothable in $X$} if it is a flat limit of
a family of finite subsets of $X$ with cardinality $\deg (Z)$ (see \cite{bb+} for a discussion of it). If $X$ is smooth (or if
$Z$ is contained in the smooth locus of $X$) $Z$ is smoothable in $X$ if and only if it is smoothable in $\PP^r$ and the
notion of smoothability in $\PP^r$ does not depend on the choice of the embedding of $Z$ in a projective space (\cite[Proposition 2.1]{bb+}).
 Let
$\rho (X)$ (resp. $\rho (X)'$, resp. $\rho (X)''$) denote the maximal integer $t>0$ such that each zero-dimensional scheme
(resp. each finite set, resp. each zero-dimensional scheme smoothable in $X$) $Z\subset X$ with $\deg (Z)=t$ is linearly
independent. Obviously $\rho (X)\le \rho (X)''\le \rho (X)'$. Since $X$ is embedded in $\PP^r$, we have $\rho (X)\ge 2$. The integers $\rho (X)$, $\rho (X)'$
and $\rho (X)''$ have been used in several papers connected to $X$-rank, the symmetric tensor rank, i.e. the additive decomposition of polynomials, and the tensor rank (\cite{bb+}, \cite{bgl}).

For any $q\in \PP^r$ the $X$-rank $r_X(q)$ of $q$ is the minimal positive integer
$t$ such that
$q\in
\langle S\rangle$ for some finite subset $S\subset X$ with $\sharp(S) =t$, where $\langle \ \
\rangle$ denote the linear span. For any positive integer
$t$ the $t$-secant variety
$\sigma _t(X)$ of $X$ is the closure in $\PP^r$ of the union of all $\langle S\rangle$ with $S$ a finite subset of $X$ with
cardinality
$t$.

The \emph{border $X$-rank} $b_X(q)$ of $q\in \PP^r$ is the minimal integer $k$ such that $q\in \sigma _k(X)$.
The \emph{generic rank} $r_{X,\mathrm{gen}}$ is the minimal integer $k>0$ such that $\sigma _k(X) =\PP^r$. There is a
non-empty open subset $U\subset \PP^r$ such that $r_X(q) =r_{X,\mathrm{gen}}$ for all $q\in U$.

In this paper we prove that if $\rho (X)''$ is large, then $X$ is linearly normal and that $\rho (X)''$ cannot be very large
for $n>1$ (Theorem \ref{xx1}). 

We prove the following results.

\begin{proposition}\label{xx1}
Assume that $X$ is a curve and $\rho (X)''  \ge \lceil (r+2)/2\rceil$. Then $X$ is linearly normal. 
\end{proposition}

\begin{proposition}\label{xx2}
Assume $n:= \dim X\ge 2$, $r_{X,\mathrm{gen}}  = \lceil (r+1)/(n+1)\rceil$ and $\rho (X)'' > \lceil (r+1)/(n+1)\rceil$. Then $X$ is linearly normal.
\end{proposition}

\begin{theorem}\label{xx2==}
Let $X\subset \PP^r$ be an integral and non-degenerate variety. Set $n:= \dim X$. We have $\rho (X)'' \ge  2\lceil
(r+1)/(n+1)\rceil$ if and only if either $r=n$ (i.e. $X =\PP^r$) or $n=1$, $r$ is odd and $X$ is a rational normal curve.
\end{theorem}

If $n=r$ we have $\rho (X)' = \rho (X)=2$. If $X$ is a rational normal curve we have $\rho (X) =\rho (X)'=r+1$. This is the
only case with $\rho (X)' =r+1$ (Lemma \ref{i3}). Theorem \ref{xx2} implies that $\rho (X)'' < 2\lceil (r+1)/(n+1)\rceil$ if $(r+1)/(n+1)\notin \ZZ$.

The example of a general linear projection in $\PP^4$ of the Veronese surface shows that in Proposition \ref{xx2} it is not
sufficient to assume that $\rho (X)'' \ge \lceil (r+2)/(n+1)\rceil$.

We point out that to get our results we only use a small family of zero-dimensional schemes, each of them with connected components of degree $1$ or $2$, but that this family
contains a complete family covering $X$: each $p\in X$ is contained in some scheme $Z$ of the family.

\section{Preliminaries}

 We recall that if $q\in \sigma _k(X)$ and $k\le \rho (X)''$ there is a zero-dimensional scheme $Z\subset X$ smoothable in $X$
and such that $\deg (Z)\le b$ and $q\in \langle Z\rangle$ (\cite[Lemma 2.6, Theorem 1.18]{bgl} and \cite[Proposition 2.5]{bb+}).

For any $q\in \PP^r$ let $\Ss (X,q)$ be the set of all $S\subset
X$ such that $|S|=r_X(q)$ and $q\in \langle S\rangle$.

\begin{remark}
Let $X\subset \PP^r$ be a smooth variety with $\dim X\le 2$. Every zero-dimensional scheme of $X$ is smoothable (\cite{f}) and hence $\rho
(X)''=\rho (X)$. Easy examples show that we may have $\rho (X) <\rho (X)'$ for a smooth curve (Examples \ref{odd1} and \ref{odd1.0}).
\end{remark}

\begin{remark}\label{0x0}
(\cite[Theorem 1.17]{bgl}) Fix $q\in \PP^r$ and $A, B\in \Ss (X,q)$. Set $x:= r_X(q)$ and assume $\rho (X)'\ge 2a$. Since $|A\cup B| \le 2a$, $A\cup B$ is linearly independent. Thus $A=B$
(\cite[Lemma 1]{bb}). Hence $|\Ss (X,q)|=1$.
\end{remark}

The following extremal case is the only result in which we are able to use only $\rho (X)'$ instead of $\rho (X)''$.

\begin{lemma}\label{i3}
The following conditions are equivalent;
\begin{enumerate}
\item $X$ is a rational normal curve;
\item $\rho (X)' =r+1$;
\item $\rho (X)'>r$;
\end{enumerate}
\end{lemma}

\begin{proof}
It is sufficient to prove that (3) implies (1). 

First assume $n=1$. Let $H\subset \PP^r$ be a general linear hyperplane. Since
$X\cap H$ is formed by $\deg (X)$ points, if $\rho (X') >r$ we have $\deg (X)=r$ and hence $X$ is a rational normal curve.

Now
assume
$n>1$. Take a general linear subspace
$V\subset
\PP^r$ with codimension $n-1$. The scheme $X\cap V$ is an integral curve spanning $V$. We have $\rho (X)'\le \rho (X\cap
V)'\le r-n+1$ by the case $n=1$ just proved.
\end{proof}

\section{The proofs}

\begin{proof}[Proof of Proposition \ref{xx1}:]
Assume that $X$ is not linearly normal. Thus there is a non-degenerate variety $Y\subset \PP^{r+1}$ such
that $X$ is an isomorphic linear projection of $Y$ from some $o\in \PP^{r+1}\setminus Y$.  Set $b:= b_Y(o)$. Each secant
variety of a curve has the expected dimension (\cite[Remark 1.6]{a}). Thus $r_{Y,\mathrm{gen}} =\lceil (r+2)/2\rceil$. Hence
$b\le \lceil (r+2)/2\rceil$. Let
$\ell :
\PP^{r+1}\setminus \{o\}\to \PP^r$ denote the linear  projection from $o$. By assumption $o\notin Y$ and $\ell _{|Y}$ is an
embedding with $\ell (Y)=X$. Let $W\subset Y$ be any zero-dimensional scheme. Since $\ell _{|Y}: Y\to X$ is an isomorphism,  $W$
is smoothable in $Y$ if and only if $\ell (W)$ is smoothable in $X$ and any degree $b$ smoothable zero-dimensional scheme is
the image of a unique degree $b$ zero-dimensional scheme. Thus $\rho (Y)'' \ge \rho (X)''$. The image in $\PP^r$ of a linear
subspace $V \subset \PP^{r+1}$ has either dimension $\dim V$ (case $o\notin V$) or dimension $\dim V -1$ (case $o\in V$).
Since $\rho (Y)'' \ge \rho (X)'' \ge \lceil (r+2)/2\rceil =r_{Y,\mathrm{gen}}$ and $b\le r_{Y,\mathrm{gen}}$, there is a
smoothable zero-dimensional scheme $W\subset Y$ such that $o\in \langle W\rangle$ and $\deg (W)=b$. Since $\ell (W)$ is not
linearly independent, we have $\rho (X)'' \le b-1$, a contradiction.
\end{proof}

\begin{proof}[Proof of Proposition \ref{xx2}:]
Assume that $X$ is not linearly normal. Thus there is a non-degenerate variety $Y\subset \PP^{r+1}$ such
that $X$ is an isomorphic linear projection of $Y$ from some $o\in \PP^{r+1}\setminus Y$.  Set $b:= b_Y(o)$ and $a:= r_{X,\mathrm{gen}} = \lceil
(r+1)/(n+1)\rceil$. Let $\ell : \PP^{r+1}\setminus \{o\}\to \PP^r$ denote the linear  projection from $o$. By assumption
$o\notin Y$ and $\ell _{|Y}$ is an embedding with $\ell (Y)=X$. As in the proof of Proposition \ref{xx1} we have $\rho
(Y)''\ge \rho (X)''$ and to get a contradiction it is sufficient to prove that $b\le \rho (X)''$. Assume $b> \rho (X)''$, i.e.
assume $b\ge a+2$. Since $b>a$, we have $o\notin \sigma _a(Y)$. Hence $\ell _{|\sigma _a(Y)}: \sigma _a(Y)\to \PP^r$ is a
finite map. Since
$\ell (\sigma _a(Y)) =\sigma _a(X)$, we get
$\dim \sigma _a(Y) =r$. Since $\dim \sigma _{a+1}(Y) > \dim \sigma _a(Y)$ (\cite[Proposition 1.3]{a}), we get $\sigma _{a+1}(Y)
=
\PP^{r+1}$. Thus
$b \le a+1$, a contradiction.\end{proof}

\begin{lemma}\label{xx4}
Assume $\rho (X)'' \ge 2\lceil (r+1)/(n+1)\rceil$. Then $X$ is not defective, $r+1\equiv 0\pmod{n+1}$ and for a general $q\in
\PP^r$ we have $|\Ss (X,q)| =1$.
\end{lemma}

\begin{proof}
Set $a:= \lceil (r+1)/(n+1)\rceil$. By Terracini's lemma (\cite[Corollary 1.11]{a}) to prove that $X$ is not defective and
that
$r+1\equiv 0\pmod{n+1}$ it is sufficient to prove that $\dim \langle T_{p_1}X\cup \cdots \cup T_{p_a}X\rangle = a(n+1)-1$ for
a general
$(p_1,\dots ,p_a)\in X^a$. Since each $p_i$ is a smooth point of $X$, each degree $2$ connected zero-dimensional scheme
$v_i\subset X$ such that $(v_i)_{\mathrm{red}} =\{p_i\}$ is smoothable. Apply \cite[Lemma 1]{b2}. Thus $X$ is not defective.
Since $r+1 =a(n+1)$ and $X$ is not defective, we have $\sigma _a(X)=\PP^r$ and $\Ss (X,q)$ is finite for a general
$q\in
\Ss (X,q)$. Fix any $q\in \PP^r$ such that $x:= r_X(q) $. We have $|\Ss (X,q)| =1$ if $\rho
(X)'\ge 2x$ by Remark \ref{0x0}. We proved that $a\ge x$.
\end{proof}

The (smooth) $n$-dimensional varieties $X\subset \PP^{2n+1}$ such that $\sigma _2(X)=\PP^{2n+1}$ and $|\Ss (X,q)|=1$ are classically called  OADP (or varieties with only one apparent double point), because projecting them from a general point of $X$ one gets a variety with a unique singular point (\cite{cmr}). They are always linearly normal
(\cite[Remark 1.2]{cmr}. In \cite{cmr} there are also older references and the classification of the smooth ones with dimension up to $3$ (\cite{r}, \cite[Theorem 7.1]{cmr}). Thus the thesis of Lemma \ref{xx4} is a generalization of this concept to the case in which $(r+1)/(n+1)$ is an integer $>2$. But the assumption ``$\rho (X)'' \ge 2\lceil
(r+1)/(n+1)\rceil$ '' of the lemma is too strong to be interesting for the classification of extremal varieties. Just assuming
$\rho '(X)>2$ excludes all $X$ containing lines and hence all smooth OADP's of dimension $2$ and $3$.

\begin{corollary}\label{xx3}
Assume $n:= \dim X\ge 2$ and $\rho (X)'' \ge 2\lceil (r+1)/(n+1)\rceil$. Then $X$ is linearly normal, non-defective, $r\equiv 0\pmod{n+1}$, $\rho (X)'' =2(r+1)/(n+1)$ and $|\Ss (X,q)|=1$ for a general $q\in \PP^r$.
\end{corollary}

\begin{proof}[Proof of Corollary \ref{xx3}:]
By Proposition \ref{xx2} it is sufficient to prove that $X$ is non-defective. Apply Lemma \ref{xx4}.
\end{proof}

\begin{proof}[Proof of Theorem \ref{xx2==}:]
Assume the existence of $X$ with $\rho (X)''\ge 2\lceil
(r+1)/(n+1)\rceil$. We may assume $n<r$, i.e. $X\ne \PP^n$.

First assume $n=1$. Lemma \ref{i3} gives that $X$ is a rational normal
curve, that
$r$ is odd and that
$\rho (X)=
\rho (X)' = r+1$.

Now assume $n\ge 2$. By Lemma \ref{xx4} $a:= (r+1)/(n+1)$ is an integer. We may assume $a\ge 2$, i.e. $r\ne n$. Fix a general $S\subset
X$
such that $|S|=a-1$. Since $S$ is general, each $p\in S$ is a smooth point of $X$. We saw in the proof of Proposition
\ref{xx2} that $V:= \langle \cup _{p\in S} T_pX\rangle$ has dimension $(a-1)(n+1)-1$. Fix $o\in X\setminus S$.

\quad \emph{Claim 1:} $o\notin V$.

\quad \emph{Proof of Claim 1:} Assume $o\in V$. We saw in the proof of Proposition \ref{xx2} that there are connected degree
$2$ zero-dimensional schemes $v_p\subset X$ such that $(v_p)_{\mathrm{red}} \{p\}$ and $o\in \langle Z\rangle$, where $Z :=
\cup _{p\in S} v_p$. Since $o\notin S$, we have $o\nsubseteq Z$. Thus the scheme $Z\cup \{o\}$ is linearly dependent.
Since $\deg (Z\cup \{o\}) = 2a-1< \rho (X)''$ and $Z\cup \{o\}$ is smoothable, we got a contradiction.
Let $\ell : \PP^r\setminus V\to \PP^n$ denote the linear projection from $V$. By Claim 1 $\ell_1: X\setminus S \to \PP^n$
is a morphism. Fix $o\in X\setminus S$ and assume the existence of $o'\in X\setminus S$ such that $o\ne o'$ and $\ell _1(o)
=\ell _1(o')$. Thus $o'\in \langle \{o\}\cup Z\rangle$. Hence $\{o,o'\}\cup Z$ is linearly dependent. The
zero-dimensional scheme $\{o,o'\}\cup Z$ is smoothable and it has  degree $2a \le \rho (X)''$, a contradiction.

Thus $\ell _1: X\setminus S \to \PP^n$ is an injective morphism between two quasi-projective varieties. Since $\PP^n$ is smooth
(it would be sufficient to assume that the target, $X$, is normal or even less (weakly normal)) and we are in characteristic
zero,
$\ell _1$ is an open map which is an isomorphism onto its image (\cite{gt}). Since $X$ is smooth at each point of $S$, $X$ is smooth. Hence the map $\ell _1: X\setminus S \to \PP?n$
extends over $S$. Since $X$ is smooth at each point of $S$, $\ell _1$ lifts to a morphism $u: \tilde{X}\to X$, where $\tilde{X}$ is the blowing-up of $X$ at all points of $S$. We first get that
$X\cong
\PP^n$ as an abstract variety and then (since any morphism $\PP^n\to \PP^n$ injective outside a finite set is an isomorphism) that $u$ does not exist when $n>1$.\end{proof}

\section{Elementary examples}

By Proposition \ref{xx1} to complete the picture for curves we need to describe the linearly normal curves with very high $\rho (X)'$, $\rho (X)$ and $\rho (X)''$.
\begin{remark}\label{oi2.0}
Let $X$ be an integral projective curve. To compute $\rho (X)''$ we recall that every Cartier divisor of $X$ is smoothable. Let $\Ff$ be any torsion free sheaf of $X$.
Duality gives $h^1(\Ff ) =\dim \mathrm{Hom}(\Ff ,\omega _X))$, (\cite[1.1 at p. 5]{ak}). Thus for any zero-dimensional
scheme
$Z\subset X$ we have
$h^1(\Ii _Z(1)) = \dim \mathrm{Hom}(\Ii _Z ,\omega _X(1))$. If $d \ge 3g-2$ we have $\rho (X) = d-2g+2$. We have $\rho (X)'
=d-2g+2$ if and only if $X$ is Gorenstein, i.e. $\omega _X$ is locally free. For lower $d$ the integers $\rho (X)$, $\rho
(X)'$ and $\rho (X)''$ depends both from the Brill-Noether theory of the special line bundles on $X$ and the choice of the
very ample line bundle $\Oo _X(1)$, not just the integers $d$ and $g$.
\end{remark}

\begin{example}\label{odd1}
Fix integers $r, a$ such that $2\le a \le r+1$. Here we prove the existence of a smooth and  non-degenerate curve
$X\subset \PP^r$ such that $\rho (X)'=\rho (X) =a$. If $a=r+1$ we know that $X$ is a rational normal curve. See the case $g=1$ and $d=r+1$ of Remark \ref{oi2.0} for
the case $a=r$. Now assume $2\le a \le r-1$. 

\quad (a) We first cover the case $2a\le r+1$. In this range we construct a smooth
rational curve $X$ with $\rho (X)=\rho (X)'=a$, but of course $X$ is not linearly normal.  Let
$Y\subset
\PP^{r+1}$ be a rational normal curve. Fix a set
$S\subset X$ such that
$|S|= a+1$ and take any $o\in \langle S\rangle$ such that $o\notin \langle S'\rangle$ for any $S'\subsetneq S$. Let $\ell
:\PP^{r+1}\setminus
\{o\}\to
\PP^r$ denote the linear projection from $o$. Since $a\ge 2$ and $\rho (Y)=r+2$, we have $o\notin Y$. Hence $\ell _{|Y}$ is a
morphism. Set $X:= \ell (Y)$.

\quad \emph{Claim 1:} $\ell _{|Y}$ is an embedding.

\quad \emph{Proof of Claim 1:} It is sufficient to prove that for any zero-dimensional scheme $A\subset Y$ with $\deg (A)\le 2$
we have $o\notin \langle A\rangle$. Assume the existence of a zero-dimensional scheme $A\subset Y$ with $\deg (A)\le 2$
and $o\in \langle A\rangle$. Since $o\notin Y$, we have $\deg (A)=2$. Since $o\notin \langle S'\rangle $ for any $S'\subsetneq
S$ and $|S|=a+1>2$, we have $A\nsubseteq S$. Since $o\in \langle A\rangle \cap \langle S\rangle$, $A\cup S$ is linearly
dependent. Since $\deg (A\cup S)\le a+3$ and $\rho (Y)=r+1$, we get a contradiction.

By Claim 1  $X$ is a smooth rational curve and $\deg (X) =r+1$.

\quad \emph{Claim 2:} We have $\rho (X) =\rho (X)'=a$.

\quad \emph{Proof of Claim 2:} Since $\ell _{|X}$ is an embedding, we have $|\ell (S)|=a+1$. Since $o\in \langle S\rangle$,
$\ell (S)$ is linear dependent and hence $\rho (X)'\le a$. Assume $\rho (X)'<a$ and take a zero-dimensional scheme $Z\subset X$
such that $\deg (Z)\le a$ and $Z$ is linearly dependent. Le $W\subset Y$ be the only scheme such that $\ell (W) =Z$.
Since $Z$ is linearly dependent, we have $o\in \langle W\rangle$. Since $\deg (W) \le a$ and $o\notin \langle S'\rangle$ for
any $S'\subsetneq S$, we have
$W\nsubseteq S$. Thus $S\cup W$ is linearly dependent. Hence $2a+1 \ge r+2$, a contradiction.

\quad (b) Now assume $2a \ge r$, $a\le r-1$ and $a\ge 2$. Fix a smooth curve $C$ of genus $g$ and a zero-dimensional scheme $A\subset X$ such that
$\deg (A)=a+1$. Since $\deg (\omega _C(A)) =2g+a+1-2 \ge 2g+1$, $\omega _C(A)$ is very ample. By Riemann-Roch we have
$h^0(\omega _C(A)) =g+a =r+1$ and $A$ is the only zero-dimensional scheme $Z\subset C$ such that $h^1(\omega _C(A-Z))>0$ and $\deg (Z)\le a+1$. Let $f: C\to \PP^r$, $r = g+a$, denote the embedding induced by the complete linear system $|\omega _C(A)$. Set $X:= f({C})$. We have $\rho (X) =a$. We have $\rho (A)' =a$ if and only if $A$ is reduced.

\quad ({c}) Take in part (a) instead of $S$ a zero-dimensional scheme $S_1$ such that $\deg (S_1)=a+1$ and $S_1$ is not
reduced. Taking the linear projection from $o$ we get an example with $\rho (X)<\rho (X)'$.
\end{example}

\begin{example}\label{odd1.0}
Fix integers $r, a$ such that $2\le a \le r+1$.  Here we prove the existence of a smooth and  non-degenerate curve
$X\subset \PP^r$ such that $\rho (X)'=\rho (X) =a$. If $a=r+1$ we know that $X$ is a rational normal curve. The case $g=1$, $d=r+1$ covers the
the case $a=r$. Now assume $2\le a \le r-1$. 

\quad (a) We first cover the case $2a\le r+1$. In this range $X$ we construct a smooth
rational curve $X$ with $\rho (X)=\rho (X)'=a$, but of course $X$ is not linearly normal.  Let
$Y\subset
\PP^{r+1}$ be a rational normal curve. Fix a set
$S\subset X$ such that
$|S|= a+1$ and take any $o\in \langle S\rangle$ such that $o\notin \langle S'\rangle$ for any $S'\subsetneq S$. Let $\ell
:\PP^{r+1}\setminus
\{o\}\to
\PP^r$ denote the linear projection from $o$. Since $a\ge 2$ and $\rho (Y)=r+2$, we have $o\notin Y$. Hence $\ell _{|Y}$ is a
morphism. Set $X:= \ell (X)$.

\quad \emph{Claim 1:} $\ell _{|Y}$ is an embedding.

\quad \emph{Proof of Claim 1:} It is sufficient to prove that for any zero-dimensional scheme $A\subset Y$ with $\deg (A)\le 2$
we have $o\notin \langle A\rangle$. Assume the existence of a zero-dimensional scheme $A\subset Y$ with $\deg (A)\le 2$
and $o\in \langle A\rangle$. Since $o\notin Y$, we have $\deg (A)=2$. Since $o\notin \langle S'\rangle $ for any $S'\subsetneq
S$ and $|S|=a+1>2$, we have $A\nsubseteq S$. Since $o\in \langle A\rangle \cap \langle S\rangle$, $A\cup S$ is linearly
dependent. Since $\deg (A\cup S)\le a+3$ and $\rho (Y)=r+1$, we get a contradiction.

By Claim 1  $X$ is a smooth rational curve and $\deg (X) =r+1$.

\quad \emph{Claim 2:} We have $\rho (X) =\rho (X)'=a$.

\quad \emph{Proof of Claim 2:} Since $\ell _{|X}$ is an embedding, we have $|\ell (S)|=a+1$. Since $o\in \langle S\rangle$,
$\ell (S)$ is linear dependent and hence $\rho (X)'\le a$. Assume $\rho (X)'<a$ and take a zero-dimensional scheme $Z\subset X$
such that $\deg (Z)\le a$ and $Z$ is linearly dependent. Le $W\subset Y$ be the only scheme such that $\ell (W) =Z$.
Since $Z$ is linearly dependent, we have $o\in \langle W\rangle$. Since $\deg (W) \le a$ and $o\notin \langle S'\rangle$ for
any $S'\subsetneq S$, we have
$W\nsubseteq S$. Thus $S\cup W$ is linearly dependent. Hence $2a+1 \ge r+2$, a contradiction.

\quad (b) Now assume $2a \ge r$, $a\le r-1$ and $a\ge 2$. Set $g:= r+1-a$. Fix a smooth curve $C$ of genus $g$ and a
zero-dimensional scheme
$A\subset X$ such that
$\deg (A)=a+1$. Since $\deg (\omega _C(A)) =2g+a+1-2 \ge 2g+1$, $\omega _C(A)$ is very ample. By Riemann-Roch we have
$h^0(\omega _C(A)) =g+a =r+1$. Let $f: C\to \PP^r$, be the embedding induced by $|\omega _X(A)|$ Set $X:= f({C})$ and $Z:=
f(A)$. By Rieman-Roch $Z$ is linearly dependent and $Z$ is the only zero-dimensional scheme $W\subset X$ such that $\deg (W)\le
a+1$ and $W$ linearly dependent. Thus $\rho (X)=a$. We have $\rho (X)'=a$ if and only if $A$ is a reduced set.

\quad ({c}) Take in part (a) instead of $S$ a zero-dimensional scheme $S_1$ such that $\deg (S_1)=a+1$ and $S_1$ is not
reduced. Taking the linear projection from $o$ we get smooth rational curves with $\rho (X)<\rho (X)'$.

\quad (d) Take $C$ as in step (b) and a zero-dimensional scheme $A_1\subset \PP^r$ such that $\deg (A_1)=a+1$ and $A_1$ is not
reduced. As in step (b) we get an embedding $X\subset \PP^r$ of $C$ such that $\rho (X)=a$ and the image of $A_1$ is the only
linearly dependent degree $a+1$ subscheme of $X$. Thus $\rho (X)'>a$.
\end{example}

\providecommand{\bysame}{\leavevmode\hbox to3em{\hrulefill}\thinspace}
\providecommand{\MR}{\relax\ifhmode\unskip\space\fi MR }
\providecommand{\MRhref}[2]{%
  \href{http://www.ams.org/mathscinet-getitem?mr=#1}{#2}
}
\providecommand{\href}[2]{#2}

\end{document}